\documentclass[a4paper,reqno,12pt]{amsart}
\usepackage{amsmath,amscd,amsfonts,amssymb}
\usepackage{mathrsfs,dsfont}
\usepackage{tikz}

\numberwithin{equation}{section}

\newcommand{\bn}{\mathbb{N}}
\newcommand{\norm}[1]{\left\lVert#1\right\rVert}
\newcommand{\seq}[1]{\left( #1\right)}
\newcommand{\floor}[1]{\left\lfloor#1\right\rfloor}
\newcommand\modu{\operatorname{mod}}
\newcommand{\countl}{\mathcal{H}}

\theoremstyle{plain}
\newtheorem{theorem}{Theorem}
\newtheorem{corollary}[theorem]{Corollary}

\title{On pair correlation and discrepancy}
\author{Sigrid Grepstad and Gerhard Larcher}
\address{Institute of Financial Mathematics and Applied Number Theory, Johannes Kepler University Linz, Altenbergerstr.\ 69, A-4040 Linz, Austria.}
\email{sigrid.grepstad@jku.at}
\address{Institute of Financial Mathematics and Applied Number Theory, Johannes Kepler University Linz, Altenbergerstr.\ 69, A-4040 Linz, Austria.}
\email{gerhard.larcher@jku.at}

\thanks{The authors are supported by the Austrian Science Fund (FWF): Projects F5505-N26 and F5507-N26, which are both part of the Special Research Program ``Quasi-Monte Carlo Methods: Theory and Applications''.}
\subjclass{Primary 11K06; Secondary 11K38}
\keywords{Pair correlation of sequences, uniform distribution modulo one, discrepancy}
\date{May 5, 2017}

\begin{document}
 
 \begin{abstract}
We say that a sequence $\seq{x_n}_{n \geq 1}$ in $[0,1)$ has Poissonian pair correlations if 
 \begin{equation*}
  \lim_{N \rightarrow \infty} \frac{1}{N} \# \left\{ 1 \leq l \neq m \leq N \, : \, \norm{x_l-x_m} < \frac{s}{N} \right\} = 2s
 \end{equation*}
for all $s>0$. In this note we show that if the convergence in the above expression is - in a certain sense - fast, then this implies a small discrepancy for the sequence $\seq{x_n}_{n \geq 1}$. As an easy consequence it follows that every sequence with Poissonian pair correlations is uniformly distributed in $[0,1)$.
 \end{abstract}
 
 \maketitle
 

\section{Introduction} 
The concept of Poissonian pair correlations for a sequence $\seq{x_n}_{n\geq1}$ in $[0,1)$ was introduced by Rudnick and Sarnak in \cite{rudnick}, and has been intensively studied by several authors over the last years (see for instance \cite{all, heath, rsz, rudzah, rudzah2}). Let $\norm{\cdot}$ denote distance to the nearest integer. We say that a sequence $\seq{x_n}_{n \geq 1}$ of real numbers in $[0,1)$ has \emph{Poissonian pair correlations} if
\begin{equation}
 \lim_{N \rightarrow \infty} \frac{1}{N} \# \left\{ 1 \leq l \neq m \leq N \, : \, \norm{x_l-x_m} < \frac{s}{N} \right\} = 2s
\label{eq:poisson}
\end{equation}
for every $s>0$. 

In this note we are concerned with the relation between the Poissonian pair correlation property and the notion of uniform distribution. We say that the sequence $\seq{x_n}_{n \geq 1}$ is \emph{uniformly distributed}, or \emph{equidistributed}, in $[0,1)$ if 
\begin{equation*}
 \lim_{N \rightarrow \infty} \frac{1}{N} \# \left\{ 1 \leq n \leq N \, : \, x_n \in [a,b) \right\} = b-a
\end{equation*}
for all $0 \leq a \leq b \leq 1$. It is well-known that uniform distribution does not necessarily imply Poissonian pair correlations. One example confirming this is the Kronecker sequence $\seq{\{ n \alpha \} }_{n \geq 1}$, which is uniformly distributed for every irrational $\alpha$, but does \emph{not} have Poissonian pair correlations for any value of $\alpha$. 
Whether the converse implication holds has until recently remained an open question: Is every sequence in $[0,1)$ with Poissonian pair correlations uniformly distributed? We answer this question in the affirmative by establishing a quantitative result connecting the speed of convergence in \eqref{eq:poisson} to the star-discrepancy $D_N^*$ of the sequence. We recall that the star-discrepancy $D_N^*$ of $\seq{x_n}_{n \geq 1}$ is defined as
\begin{equation*}
  D_N^* = \sup_{0 \leq a \leq 1} \left| \frac{1}{N} \cdot A_N\left( [0,a)\right) - a \right|,
\end{equation*}
where $A_N([0,a)) := \# \{ 1 \leq n \leq N \, : \, x_n \in [0,a) \}$, and that $\seq{x_n}_{n \geq 1}$ is uniformly distributed in $[0,1)$ if and only if $\lim_{N\rightarrow \infty} D_N^* = 0$ (see for example \cite{kuipers}).

The main result of this paper is the following. 
\begin{theorem}
 Let $\seq{x_n}_{n\geq 1}$ be a sequence in $[0,1)$, and suppose that there exists a function $F : \bn \times \bn \rightarrow \mathbb{R}^+$ which is monotonically increasing in its first argument, and which satisfies 
\begin{equation}
 \max_{s=1, \ldots , K} \left| \frac{1}{2s} \# \left\{ 1 \leq l \neq m \leq N \, : \, \norm{x_l-x_m} < \frac{s}{N} \right\} - N \right| \leq F(K,N)
\label{eq:Fcond}
 \end{equation}
for all $N \in \bn$ and all $K\leq N/2$. One can then find an integer $N_0>0$ such that for $N\in \bn$, $N\geq N_0$, and arbitrary $K$ satisfying
  \label{thm:main}
\begin{equation}
\min \left(\frac{1}{2} N^{2/5}, \frac{N}{F\left(K^{2} , N\right)}\right) \leq K \leq N^{2/5},
\label{eq:min}
\end{equation}
we have
\begin{equation*}
ND_{N}^{*} \leq 5 \cdot \max \left(N^{4/5}, \sqrt{N \cdot F \left(K^{2}, N\right)}\right)
\end{equation*}
where $D_N^*$ is the star-discrepancy of $\seq{x_n}_{n \geq 1}$.
\end{theorem}

The next result is an easy consequence of Theorem \ref{thm:main}.
\begin{corollary}
If the sequence $\seq{x_n}_{n\geq 1}$ in $[0,1)$ has Poissonian pair correlations, then it is uniformly distributed.\footnote{Simultaneously with our proof, another elegant proof of this result was given by Aistleitner et.al. in \cite{alp}. However, their approach is less elementary and does not provide the quantitative bound on the star discrepancy given by Theorem \ref{thm:main}.}
 \label{cor:main}
\end{corollary}
\begin{proof}
 Suppose that $\seq{x_n}_{n \geq 1}$ has Poissonian pair correlations, and fix any $\varepsilon>0$. We then have
  \begin{equation*}
   \max_{s=1, \ldots, \floor{1/\varepsilon^{5}}} \left| \frac{1}{2s} \# \left\{ 1\leq l \neq m \leq N \, : \, \norm{x_l - x_m} < \frac{s}{N} \right\} - N\right| \leq \varepsilon N ,
  \end{equation*}
 for all sufficiently large $N\geq N(\varepsilon)$. Hence, we may construct a function $F$ satisfying \eqref{eq:Fcond} where $F(L,N)=\varepsilon N$ for $N\geq N(\varepsilon)$ and $L\leq 1/\varepsilon^{5}$. Without loss of generality, we may assume that $N(\varepsilon) \geq 1/\varepsilon^5$. If we fix $K:=\floor{1/\varepsilon^2}$, then for $N\geq N(\varepsilon)$ we have
 \begin{equation*}
  \frac{N}{F(K^2,N)}=\frac{N}{\varepsilon N} = \frac{1}{\varepsilon} \leq K \leq N^{2/5},
 \end{equation*}
and accordingly $K$ satisfies \eqref{eq:min}. By Theorem \ref{thm:main} it thus follows that
 \begin{equation*}
 D_{N}^{*} \leq \frac{5}{N} \cdot \max \left(N^{4/5}, N \varepsilon \right) = 5 \sqrt{\varepsilon}
 \end{equation*}
 for $N\geq N_0$ (where in particular $N_0\geq N(\varepsilon)$). 
\end{proof}

\section{Proof of Theorem \ref{thm:main}}
For a fixed pair of integers $(N,K)$, where $K$ satisfies \eqref{eq:min}, we introduce the notation
\begin{equation*}
 H(N,K) := 5 \cdot \max \left(N^{4/5}, \sqrt{N \cdot F \left(K^{2}, N\right)}\right).
\end{equation*}
Aiming for a proof by contradiction, we assume that $ND_N^* > H(N,K)$ for infinitely many pairs $(N,K)$. That is, there exist integers $1 < N_1 < N_2 < \ldots$ and corresponding integers $K_1, K_2, \ldots$ satisfying \eqref{eq:min}, as well as real numbers $B_1, B_2, \ldots  \in (0,1)$, such that either
\begin{equation}
 \# \left\{ 1 \leq n \leq N_j \, : \, x_n \in [0,B_j) \right\} - N_jB_j > H(N_j, K_j)
\label{eq:biggerH}
 \end{equation}
for every $j$, or
\begin{equation}
  \# \left\{ 1 \leq n \leq N_j \, : \, x_n \in [0,B_j) \right\} - N_jB_j < -H(N_j, K_j)
\label{eq:smallerH}
\end{equation}
for every $j$. We assume in what follows that \eqref{eq:biggerH} holds (the case when \eqref{eq:smallerH} holds is treated analogously). Note that \eqref{eq:biggerH} implies
\begin{equation}
N_{j} - N_{j}B_{j} - H\left(N_{j}, K_j \right) > 0.
\label{eq:raute}
\end{equation}
Let $N:= N_j$, $K:=K_j$, $B:=B_j$ and $H:=H(N_j, K_j)$ for some fixed $j$.
We now consider the distribution of the points $x_n$ into subintervals of $[0,1)$ of length $K/N$. Let 
\begin{equation*}
 A_i := \# \left\{ 1 \leq n \leq N \, : \, x_n \in \left[ i \cdot \frac{K}{N}, (i+1) \cdot \frac{K}{N} \right) \right\}
\end{equation*}
for $i=0,1,\ldots , \floor{N/K}-1$, and let
\begin{equation*}
 A_{\floor{N/K}}:= \# \left\{ 1 \leq n \leq N \, : \, x_n \in \left[ \floor{\frac{N}{K}}\cdot \frac{K}{N}, 1 \right) \right\}.
\end{equation*}
Moreover, for arbitrary positive integers $l$, let
\begin{equation*}
 A_l := A_{l \modu (\floor{N/K}+1)}.
\end{equation*}
If we introduce the notation
\begin{equation*}
 \countl_L := \# \left\{ 1 \leq l \neq m \leq N \, :\, \norm{x_l-x_m} < \frac{LK}{N}\right\}
\end{equation*}
for $L=1,2, \ldots , K$, then
\begin{equation}
\left| \frac{1}{2LK} \countl_L - N \right| \leq F(K^2,N).
 \label{eq:HFrel}
\end{equation}
We have that 
\begin{equation*}
 \begin{aligned}
  \countl_L &\geq \sum_{i=0}^{\floor{N/K}} \left( A_i(A_i-1) +2A_i(A_{i+1}+\cdots + A_{i+L-1}) \right)\\
  &= \sum_{i=0}^{\floor{N/K}} \left( (A_i+ \ldots + A_{i+L-1})^2-(A_{i+1}+\ldots +A_{i+L-1})^2 \right) -N \\
  &=: 2LKN \cdot \gamma_L - N ,
 \end{aligned}
\end{equation*}
where
\begin{equation*}
 \gamma_L = \frac{1}{2LKN} \sum_{i=0}^{\floor{N/K}} \left( (A_i+ \cdots +A_{i+L-1})^2-(A_{i+1}+ \cdots + A_{i+L-1})^2 \right).
\end{equation*}
Thus, we get 
\begin{equation}
 \frac{1}{2LKN} \cdot \countl_L \geq \gamma_L - \frac{1}{2LK} .
\label{eq:countlbound}
 \end{equation}
 
Now consider 
\begin{equation}
 \Gamma_K := \min_{x_1, \ldots , x_N} \max_{L=1,2\ldots , K} \gamma_L , 
\label{eq:Gammadef}
\end{equation}
where by $\min_{x_1, \ldots , x_N}$ we mean the minimum over all configurations of the points $x_1, \ldots , x_N$ satisfying \eqref{eq:biggerH}. If we define
\begin{equation*}
 Z_L:= \frac{1}{2LKN} \sum_{i=0}^{\floor{N/K}} \left( A_i+A_{i+1}+\cdots + A_{i+L-1} \right)^2,
\end{equation*}
then
\begin{equation*}
 \gamma_L = Z_L - \frac{L-1}{L} \cdot Z_{L-1} , 
\end{equation*}
and thus
\begin{equation*}
 \Gamma_K = \min_{x_1, \ldots , x_N} \max \left( Z_1, Z_2-\frac{1}{2}Z_1, \cdots , Z_K - \frac{K-1}{K} Z_{K-1} \right) .
 \end{equation*}
We have 
\begin{equation*}
\max \left( Z_1, Z_2-\frac{1}{2}Z_1, \cdots , Z_K - \frac{K-1}{K} Z_{K-1} \right) \geq \frac{2}{K+1} Z_K.
\end{equation*}
To see this, assume to the contrary that $Z_1$ and $Z_L-(L-1)Z_{L-1}/L$ are all less than $2Z_K/(K+1)$. Then by succesive insertions we get the contradiction $Z_K < Z_K$.
Hence, we have
\begin{equation}
 \Gamma_K \geq \min_{x_1, \ldots , x_N} \frac{2}{K+1} \cdot Z_K . 
 \label{eq:GB1} 
\end{equation}
Let us now estimate
 \begin{equation*}
  \min_{x_1, \ldots , x_N} Z_K = \frac{1}{2K^2N} \min_{A_0, A_1, \ldots , A_{\floor{N/K}}} \sum_{i=0}^{\floor{N/K}} (A_i+A_{i+1}+\cdots + A_{i+K-1})^2,
 \end{equation*}
where the minimum on the right hand side is taken over all possible values of $A_0, A_1, \ldots , A_{\floor{N/K}}$ provided that the points $x_1, \ldots , x_N$ satisfy \eqref{eq:biggerH}. 
By definition, we have $A_0+\cdots +A_{\floor{N/K}} = N$. Introducing the notation $G_i=A_i+A_{i+1}+\cdots +A_{i+K-1}$, we thus get
 \begin{equation}
  \sum_{i=0}^{\floor{N/K}} G_i = K \cdot \sum_{i=0}^{\floor{N/K}} A_i = KN.
 \label{eq:Gsumtot}
  \end{equation} 
Moreover, by invoking condition \eqref{eq:biggerH} on the distribution of $x_1, \ldots , x_N$, we have
\begin{equation}
 \sum_{i=-K+1}^{\floor{NB/K}} G_i \geq K \sum_{i=0}^{\floor{NB/K}} A_i \geq K(NB+H),
\label{eq:Gsumbigger}
\end{equation}
and consequently
\begin{equation}
\sum_{i= \floor{NB/K}+1}^{\floor{N/K}-K} G_i \leq K\left(N(1-B)-H\right) .
 \label{eq:Gsumsmaller}
\end{equation}
We get
\begin{equation}
 \min_{x_1, \ldots, x_N} Z_K  \geq \frac{1}{2K^2N} \min_{G_0, G_1, \ldots , G_{\floor{N/K}}} \sum_{i=0}^{\floor{N/K}} G_i^2 , 
\label{eq:GB2}
 \end{equation}
where the minimum on the right hand side is taken over all positive reals $G_0, G_1, \ldots , G_{\floor{N/K}}$ satisfying \eqref{eq:Gsumtot} -- \eqref{eq:Gsumsmaller}. It is an easy exercise to verify that this minimum is attained when 
\begin{equation*}
 G_i = \frac{K(NB+H)}{K+\floor{NB/K}} \quad \text{ for } i=-K+1, \ldots , \floor{\frac{NB}{K}} ,
\end{equation*}
and 
\begin{equation*}
 G_i= \frac{K\left( N(1-B) - H\right)}{\floor{N/K}-K-\floor{NB/K}} \quad \text{ for } i=\floor{\frac{NB}{K}}+1 , \ldots , \floor{\frac{N}{K}}-K.
\end{equation*}
Note that since $K \leq N^{2/5}$ and $H \geq 5 N^{4/5}$, we have $K^{2} \leq H/5$, and hence by \eqref{eq:raute} both the numerator and the denominator of these $G_{i}$ are positive.
Thus, we get
\begin{equation}
\begin{aligned}
 &\frac{1}{2K^2N} \min_{G_0, G_1, \ldots , G_{\floor{N/K}}} \sum_{i=0}^{\floor{N/K}} G_i^2
 \\
 &\geq \frac{1}{2K^2N} \left( \frac{ K^2(NB+H)^2}{K+\floor{NB/K}} + \frac{K^2\left(N(1-B)-H \right)^2}{\floor{N/K}-K-\floor{NB/K}}\right) \\
 &\geq \frac{K}{2} \left( 1+\frac{H^2}{2N^2}\right)
\end{aligned}
\label{eq:GB3}
\end{equation}
for all $N > N_0$. For the final inequality in \eqref{eq:GB3}, we have again used that $H\geq 5N^{5/4}$ and $K^{2} \leq H/5$.

Finally, by combining \eqref{eq:GB3}, \eqref{eq:GB2} and \eqref{eq:GB1}, we find the lower bound
\begin{equation*}
 \Gamma_K \geq \frac{K}{K+1} \left( 1+\frac{H^2}{2N^2} \right).
\end{equation*}
From the definition \eqref{eq:Gammadef} of $\Gamma_K$ and \eqref{eq:countlbound}, it follows that 
\begin{equation*}
 \max_{L=1,\ldots , K} \frac{1}{2LKN} \countl_L > \Gamma_K - \frac{1}{2K} \geq 1+ \frac{H^2}{4N^2} - \frac{2}{K} ,
\end{equation*}
and recalling \eqref{eq:HFrel}, we get
\begin{equation*}
 \frac{1}{N} F(K^2, N) + 1 \geq \max_{L=1,\ldots , K} \frac{1}{2LKN} \countl_L > 1+ \frac{H^2}{4N^2} - \frac{2}{K} .
\end{equation*}
This implies that
\begin{equation*}
\begin{aligned}
&H^2 < \frac{8N^2}{K} + 4NF(K^2,N)
\\
&\leq 12 \max \left(\frac{N^{2}}{K}, N F \left(K^{2}, N\right)\right) \\
&< 25 \max \left(N^{8/5}, N F \left(K^{2}, N\right)\right) = H^{2},
\end{aligned}
\end{equation*}
which is a contradiction. Thus, our assumption \eqref{eq:biggerH} must be incorrect, and the proof of Theorem \ref{thm:main} is complete. (Note that the last inequality above is trivially true if $N^2/K \leq NF(K^2,N)$; In the opposite case we have $K<N/F(K^2,N)$, and by the condition \eqref{eq:min} imposed on $K$ we then get $K\geq N^{2/5}/2$, and consequently $N^2/K \leq 2N^{8/5}$.)

\section*{Acknowledgement}
The authors thank an anonymous reviewer who pointed out an inaccuracy in the first version of the paper. His helpful comments led to the current, slightly stronger version of Theorem~\ref{thm:main}.


\end{document}